\documentclass[reqno]{amsart}
\usepackage{amsmath, amsthm, amssymb, mathabx, enumerate}
\begin{document}
\newtheorem{theo}{Theorem}
\newtheorem{defin}[theo]{Definition}
\newtheorem{rem}[theo]{Remark}
\newtheorem{lemme}[theo]{Lemma}
\newtheorem{cor}[theo]{Corollary}
\newtheorem{prop}[theo]{Proposition}
\newtheorem{exa}[theo]{Example}
\newtheorem{exas}[theo]{Examples}
%
%
\thanks{This work was funded by an NSERC grant.}
\title[Generalized fountain theorem and application]{Generalized fountain theorem and application to strongly indefinite semilinear problems}

\author[C. J. Batkam]{Cyril Jo\"el Batkam}
\address{Cyril Jo\"el Batkam \newline
D\'epartement de math\'ematiques,
\newline
Universit\'e de Sherbrooke,
\newline
Sherbrooke, (Qu\'ebec),
\newline
J1K 2R1, CANADA.
\newline {\em Fax:} (819)821-7189.}
\email{cyril.joel.batkam@Usherbrooke.ca}
\author[F. Colin]{Fabrice Colin}
\address{Fabrice Colin \newline
Department of Mathematics and Computer Science,
\newline
Laurentian University,
\newline
Ramsey Lake Road, Sudbury, (Ontario),
\newline
P3E 2C6, CANADA.
\newline {\em Fax:} (705)673-6591.}
\email{fcolin@cs.laurentian.ca}
\begin{abstract}
By using the degree theory and the $\tau-$topology of Kryszewski and Szulkin, we establish a version of the Fountain Theorem for strongly indefinite functionals.
The abstract result will be applied for studying the existence of infinitely many solutions of two strongly indefinite semilinear problems including the semilinear Schr\"{o}dinger equation.
\end{abstract}
\maketitle
\section*{Introduction}
%
The Fountain Theorem established by T. Bartsch \cite{B} and M. Willem \cite{W} is a powerful tool for studying the existence of infinitely many critical points of some indefinite functionals, which are symmetric in the sense of
admissibility posed by Bartsch \cite{B} (see also \cite{B-W, W}). This theorem has been used by several authors to study the existence of infinitely many large energy solutions of various semilinear problems, see for instance \cite{B, B-W, Y-H, Ton, Sev, I, W}. We also mention the paper of W. Zou \cite{Zou}, where a variant Fountain Theorem was established without $(PS)-$type assumption.
\paragraph{} \indent The aim of this paper is to generalize the Fountain Theorem in order to apply it to a wider class of indefinite functionals, specially to strongly indefinite functionals; that is functionals of the form $\phi(u)=\frac{1}{2}\big<Lu,u\big>-\psi(u)$ defined on a Hilbert space $X$, where $L:X\rightarrow X$ is a selfadjoint operator with negative and positive eigenspace both infinite-dimensional. The study of such functionals is motivated by a number of problems from mathematical physic. They arise for example in the study of periodic solutions of the one-dimensional wave equation, in the study of periodic solutions of Hamiltonian systems, or in the existence theory for systems of elliptic equations.
\\
\paragraph{}
The paper is organized as follows: In the first section, the main abstract results are outlined and future applications are discussed while in section $2$, the degree theory of Kryszewski and Szulkin is introduced and a Borsuk-Ulam type theorem for admissible maps is stated. In section $3$, we construct a deformation which will be very helpful to establish the new fountain theorem in section $4$. Sections $5$ and $6$ are dedicated to the application of the abstract result to a nonlinear stationary Schr\"{o}dinger equation and to a noncooperative system of elliptic equations, respectively.

 \section{Main Results}
 Let $X$ be a real Hilbert space with inner product $(\cdot)$, norm $\|\cdot\|$ and decomposition $X=Y\oplus Z$, where $Y$ is closed and separable, and $Z=\overline{\oplus_{j=0}^\infty\mathbb{R} e_j}$. Define for $k\geq2$ and $0<r_k<\rho_k$,
 \begin{equation*}
    \displaystyle Y_k:=Y\oplus(\oplus_{j=0}^k\mathbb{R} e_j),\,\ Z_k:=\overline{\oplus_{j=k}^\infty\mathbb{R} e_j},
 \end{equation*}
 \begin{equation*}
    \displaystyle B_k:=\big\{u\in Y_k\,\ \bigl|\,\ \|u\|\leq\rho_k\big\}\,\ and \,\ N_k:=\big\{u\in Z_k\,\ \bigl|\,\ \|u\|=r_k \big\}.
 \end{equation*}
Consider on $X=Y\oplus Z$ the $\tau-$topology introduced by Kryszewski and Szulkin in \cite{K-S} (see also \cite{W}, Chapter $6$), and let $\varphi:X\rightarrow\mathbb{R}$ be a $\mathcal{C}^1-$functional such that $\varphi$ is $\tau-$upper semicontinuous, and $\nabla\varphi$ is weakly sequentially continuous. We are interested in the obtention of critical points of $\varphi$ when the latter has the following linking geometry
\begin{equation*}
    \sup_{\substack{u\in\partial B_k}}\varphi(u)<\inf_{\substack{u\in N_k}}\varphi(u).
\end{equation*}
In order to achieve this goal, we assume that $\varphi$ is invariant under an admissible action of a finite group $G$ (see Definition $\ref{defadmissible}$ below), which the antipodal action of $\mathbb{Z}_2$ is a particular case. More precisely, we will introduce the following assumption:
\begin{enumerate}
  \item [$(A_1)$] A  finite group  $G$ acts isometrically and $\tau$-isometrically on $X$, and the action of $G$ on every subspace of $X$ is admissible (in the sense of Definition $\ref{defadmissible}$).
\end{enumerate}
Roughly speaking, the above-mentioned admissibility authorizes an extension of the Borsuk-Ulam theorem for a certain class of functions conveniently called $\sigma-$admissible later in the text. By taking advantage of the notion of admissible group, we show that $B_k$ and $N_k$ link in the following sense: If $\gamma:B_k\rightarrow X$ is $\tau-$continuous and equivariant, $\gamma_{|\partial B_k}=id$, and every $u\in B_k$ has a $\tau-$neighborhood $N_u$ in $Y_k$ such that $(id-\gamma)\bigl(N_u\cap int(B_k)\bigr)$ is contained in a finite-dimensional subspace of $X$, then $\gamma(B_k)\cap N_k\neq\emptyset$. Since in our purpose $Y$ can be infinite-dimensional, the Brouwer degree \big(which is usually used in the finite-dimensional case\big) will be replaced with the Kryszewski-Szulkin degree (see Definition \ref{deg} below) in the establishment of the above linking. By constructing an equivariant deformation with the flow of a certain gradient vector field, we finally show that if in addition
\begin{equation*}
    \sup_{\substack{B_k}} \varphi(u)<\infty,
\end{equation*}
 then there is a sequence $(u_k^n)\subset X$ such that
\begin{equation*}
    \varphi'(u_k^n)\rightarrow0,\,\,\,\,\,\,\,\ \varphi(u_k^n)\rightarrow \inf_{\substack{\gamma\in\Gamma_k}}\sup_{\substack{u\in B_k}}\varphi(u_k^n) \,\,\,\ \textnormal{as} \,\ n\rightarrow\infty,
\end{equation*}
where $\Gamma_k$ is a class of maps $\gamma:B_k\rightarrow X$ defined in Theorem $\ref{tm}$. The abstract result will be used to study the existence of infinitely many large energy solutions of the two following strongly indefinite semilinear problems.

\subsection{Semilinear Schr\"{o}dinger equation}
\indent For the first application we consider the following nonlinear stationary Schr\"{o}dinger equation
\begin{equation*}
    (\mathcal{P}) \,\,\,\,\,\,\,\,\,\,\ \left\{
                                          \begin{array}{ll}
                                            -\Delta u+V(x)u=f(x,u), & \hbox{} \\
                                          \,\,\  u\in H^1(\mathbb{R}^N), & \hbox{}
                                          \end{array}
                                        \right.
\end{equation*}
under the following assumptions:
\begin{enumerate}
 \item [$(V_0)$] The function $V:{\mathbb R}^N\rightarrow{\mathbb R}$ is continuous and $1-$periodic in $x_1,...,x_N$ and $0$ lies in a gap of the spectrum of $-\Delta+V$.
 \item [$(f_1)$] The function $f:{\mathbb R}^N\times{\mathbb R}\rightarrow{\mathbb R}$ is continuous and $1-$periodic with respect to each variable $x_j, j=1,...,N$.
 \item [$(f_2)$] There is a constant $c>0$ such that
 $$|f(x,u)|\leq c(1+|u|^{p-1})$$
 for all $x\in {\mathbb R}^N$ and $u\in{\mathbb R}$, where $p>2$ for $N=1,2$ and $2<p<2N/N-2$ if $N\geq3$.
 \item [$(f_3)$] $f(x,u)=\circ(|u|)$ uniformly with respect to $x$ as $|u|\rightarrow0$.
 \item [$(f_4)$] There exists $\gamma>2$ such that for all $x\in {\mathbb R}^N$ and $u\in {\mathbb R}\backslash\{0\}$
 $$0<\gamma F(x,u)\leq uf(x,u),$$ where $F(x,u):=\int_0^u f(x,s)ds$.
 \item [$(f_5)$] For all $x\in {\mathbb R}^N$ and $u\in{\mathbb R}$, $f(x,-u)=-f(x,u)$.
 \end{enumerate}
 The natural energy functional associated to $(\mathcal{P})$ is given by
 \begin{equation*}
    \varphi(u):=\int_{\mathbb{R}^N}\bigl(|\nabla u|^2+V(x)u^2-F(x,u)\bigr)dx,\,\,\, u\in H^1(\mathbb{R}^N).
 \end{equation*}
By $(V_0)$ the Schr\"{o}dinger operator $-\Delta+V$ \big(in $L^2(\mathbb R^N)$\big) has purely continuous spectrum, and the space $H^1(\mathbb R^N)$ can be decomposed into $H^1(\mathbb R^N)=Y\oplus Z$ such that the quadratic form
\begin{equation*}
    u\in H^1(\mathbb R^N)\mapsto \int_{\mathbb R^N}\bigl(|\nabla u|^2+V(x)u^2\bigr)dx,
\end{equation*}
is negative and positive definite on $Y$ and $Z$ respectively. Both $Y$ and $Z$ are infinite-dimensional, so the functional $\varphi$ is strongly indefinite. This case has been of large interest in the last two decades, existence and multiplicity results have been obtained by different methods, see for instance \cite{K-S, Al-Li1, Al-Li2, S-W, P, W-Tro, P-P}. We prove in this paper that, under the above assumptions, $(\mathcal{P})$ has infinitely many large energy solutions.
\subsection{Noncooperative elliptic system}
\indent As the second application we study the following system:
\begin{equation*}
 (\mathcal{P}') \,\,\,\,\,\,\,\ \left\{                           \begin{array}{ll}
                                          \,\,\  \Delta u=H_u(x,u,v) \,\ \text{in} \,\ \Omega, & \hbox{} \\
                                          -\Delta v=H_v(x,u,v) \,\ \text{in} \,\ \Omega, & \hbox{} \\
                                          \,\,\  u=v=0 \,\ \text{on} \,\ \partial\Omega, & \hbox{}
                                          \end{array}
                                        \right.
\end{equation*}
where $\Omega$ is an open bounded subset of $\mathbb{R}^N$. It is well known that $(\mathcal{P}')$ has a variational structure, and the associated Euler-Lagrange functional is given by
\begin{equation*}
    \Phi(u,v):=\int_{\mathbb{R}^N}\Big(\frac{1}{2}|\nabla v|^2-\frac{1}{2}|\nabla u|^2-H(x,u,v)\Big)dx \,\,\,\,\ u,v\in H_0^1(\Omega).
\end{equation*}
By taking advantage of the new Fountain Theorem, we obtain infinitely many large energy solutions of $(\mathcal{P}')$, provided the following conditions are satisfied:\\
\indent $(h_1)$ \,\,\,\ $H\in\mathcal{C}^1(\overline{\Omega}\times\mathbb{R}\times\mathbb{R})$ and $H(x,0,0)=0$, $\forall x\in \overline{\Omega}.$ \\
 \indent  $(h_2)$ \,\,\,\ $\exists c>0$ such that
 \begin{equation*}
   |H_u(x,u,v)|+|H_v(x,u,v)|\leq c\big(1+|u|^{p-1}+|v|^{p-1}\big),
\end{equation*}
for all $x\in \mathbb{R}^N$ and $u,v\in\mathbb{R}$, where $p>2$ for $N=1,2$ and $2<p<2N/N-2$ if $N\geq3$.\\
\indent  $(h_3)$ \,\,\,\ $ 0< p H(x,u,v)\leq uH_u(x,u,v)+vH_v(x,u,v), \,\ for \,\ (u,v)\neq(0,0), \,\ \forall x\in\overline{\Omega}.$ \\
\indent $(h_4)$ \,\,\,\,\,\ $H(x,-u,-v)=H(x,u,v), \,\,\,\ \forall u,v\in\mathbb{R}, \forall x\in\overline{\Omega}$.\\
\paragraph{} The solutions of $(\mathcal{P}')$ represent the steady state solutions of reaction-diffusion
systems which are derived from several applications, such as mathematical biology
or chemical reactions (see \cite{Lupo} and reference therein). Recently the existence and multiplicity of solutions for noncooperative elliptic systems of the form $(\mathcal{P}')$ have been proved by several authors, see for instance \cite{B-C, F, Hira, F-D, Lupo} and references therein.
%
\section{Kryszewski-Szulkin degree theory}
Let $(\mathcal{H},(\cdot),\|\cdot\|)$ be a separable Hilbert space. Let $(b_{j})_{j\geq0}$ be a total orthonormal sequence in $\mathcal{H}$ and let
\begin{equation*}
    \|u \|_1 \,  :=  \, \sum_{j=0}^{\infty} \frac{1}{2^{j+1}} \bigl|(u, \, b_{j})  \bigr|, \,\,\,\,\ u\in \mathcal{H}.
\end{equation*}
We will denote by $\sigma$ the topology generated by the norm $\| \cdot \|_1$. Clearly $\|y\|_1\leq\|y\|$ for every $y\in \mathcal{H}$, and moreover if $(y_n)$ is a bounded sequence in $\mathcal{H}$ then
\begin{equation*}
    y_{n} \rightharpoonup y \Longleftrightarrow  y_{n}\stackrel{\sigma}{\rightarrow} y.
\end{equation*}
\indent Let $U$ be an open bounded subset of $\mathcal{H}$ such that $\overline{U}$ is $\sigma-$closed. The following definitions are due to Kryszewski and Szulkin (see \cite{W}).
\begin{defin}\label{sigmaadm}
{\em   A map $f \, : \, \overline{U} \rightarrow \mathcal{H}$
is said to be $\sigma$-admissible if it meets the two following conditions.
\begin{enumerate}
\item{$f$ is $\sigma-$continuous,}
\item{each point $u \in U$ has a $\sigma-$neighborhood $N_{u}$ such that $(id -f)(N_{u} \cap U)$ is contained
in a finite-dimensional subspace of $\mathcal{H}.$}
\end{enumerate} }
\end{defin}
\begin{defin}\label{sigmaadmhom}
An application $h:[0,1]\times U\rightarrow \mathcal{H}$ is a $\sigma$-admissible homotopy if:
\begin{itemize}
  \item [(a)] $h^{-1}(0)\cap\big([0,1]\times \partial U\big)=\emptyset$,
  \item [(b)] $h$ is $\sigma$-continuous,
  \item [(c)] Every $(t,u)\in [0,1]\times U$ has a $\sigma$-neighborhood  $N_{(t,u)}$ such that $\big\{s-h(\tau,s)/(\tau,s)\in N_{(t,u)}\cap ([0,1]\times U)\big\}$ is contained in a finite-dimensional subspace of $\mathcal{H}$.
      \end{itemize}

\end{defin}
If $f$ is a $\sigma$-admissible map such that $0\notin f(\partial U)$, then $f^{-1}(0)$ is $\sigma-$closed and consequently, $\sigma-$compact.
Let us consider the covering of $f^{-1}(0)$ by sets $N_{u}$ which are $\sigma-$open neighborhoods of $u \in f^{-1}(0)$ such that $(id-f)(N_{u} \cap U)$ is contained in a finite-dimensional subspace of $\mathcal{H}.$ So, we can find $m$ points $u_{1}, \, ..., \, u_{m} \in f^{-1}(0) $ such that $ f^{-1}(0) \subset V:=\bigcup_{n=1}^{m} N_{u_{n}} \cap U$ and $(id-f)(N_{u_k} \cap U)$ is contained in a finite-dimensional subspace $F$ of $\mathcal{H}$, $\forall k=1, \,\ldots ,m$. The contraction and the excision properties of the Brouwer's degree imply that $deg_B(f_{|V\cap F},V\cap F)$ does not depend on the choice of $V$ and $F$, where $deg_B$ is the Brouwer's degree. This leads to the following definition due to Kryszewski and Szulkin:
\begin{defin}[Kyszewski-Szulkin degree]\label{deg}
Let $f$ be a $\sigma$-admissible map such that $0\notin f(\partial U)$. The degree of $f$ is defined by
\begin{equation*}
    deg(f,U):=deg_B(f_{|V\cap F},V\cap F),
\end{equation*}
 where $V$ and $F$ are defined above.
\end{defin}
\begin{prop}\label{deg1}\quad
\begin{enumerate}[(i)]
  \item If $y\in U$ then $deg(id-y,U)=1$.
  \item Let $f$ be a $\sigma$-admissible map such that $0\notin f(\partial U)$. If $deg(f,U)\neq0$ then there exists $u\in U$ such that $f(u)=0$.
  \item If $h$ is a $\sigma$-admissible homotopy, then $deg(h(t,\cdot),U)$ does not depend on the choice of $t$.
\end{enumerate}
\end{prop}
For the proof of Proposition $\ref{deg1}$, we refer the reader to \cite{W}.\\
 The following theorem was inspired by \cite{K-S} Theorem $2.4(iv)$.
\begin{theo}[Borsuk-Ulam theorem for admissible maps]\label{bu}
Let  $U$ be an
open bounded symmetric neighborhood of $0$ in $\mathcal{H}$ such that
$\overline{U}$ is $\sigma-$closed. Let $f:\overline{U}\rightarrow \mathcal{H}$ be a $\sigma$-admissible odd map.
If $f(\overline{U})$ is contained in a proper subspace of $\mathcal{H}$, then there exists $u_0\in \partial U$ such that $f(u_0)=0$.
\end{theo}
\begin{proof}
Assume by contradiction that $f^{-1}(0)\cap \partial U=\emptyset$. Since $f$ is odd, we may assume that $V$ is a symmetric (i.e $-V=V$). Let $F$ be a proper subspace of $\mathcal{H}$ such that $f(\overline{U})\subset F$, and let  $z\in \mathcal{H}\backslash F$.
Define $h:[0,1]\times \overline{U}\rightarrow \mathcal{H}$ by $h(t,u):=f(u)-tz$.
One can easily verify that $h$ is a $\sigma$-admissible homotopy. By Proposition $\ref{deg1} (iii)$, $deg(h(0,\cdot),U)=deg(h(1,\cdot),U)$. The classical Borsuk theorem implies that $deg(h(0,\cdot),U)=deg(f,U)\neq0$, so $deg(h(1,\cdot),U)=deg(f-z,U)\neq0$. It then follows from Proposition $\ref{deg1} (ii)$ that there exists $u_0\in U$ such that $z=f(u_0)$, which is a contradiction since $z\in \mathcal{H}\backslash f(\overline{U})$.
\end{proof}
\indent Now we need to precise what kind of symmetries we will consider in the sequel.
 We recall that the action of a topological group $G$ on a normed vector space $(E,\|\cdot\|)$ is a continuous map $G\times E\rightarrow E$, $(g,u)\mapsto gu$, such that:
\begin{enumerate}[(1)]
  \item $eu=u$, $\forall u\in E$.
  \item $h(gu)=(hg)u$, $\forall u\in E,\,\ \forall h,g\in G$.
  \item The map $u\mapsto gu$ is linear for every $g\in G.$
\end{enumerate}
The action of $G$ is isometric if $\|gu\|=\|u\|$, $\forall u\in E,g\in G$. A subset $F$ of $E$ is invariant if $gF=F$ for every $g\in G$. A functional $\varphi:E\rightarrow \mathbb{R}$ is invariant if $\varphi(gu)=\varphi(u)$, for every $u\in E$ and $g\in G$. A map $f:X\rightarrow X$ is equivariant if $f(gu)=gf(u)$, for every $u\in E$ and $g\in G.$
\begin{defin}\label{defadmissible}
Let $G$ be a finite group acting on $\mathcal{H}$. The action of $G$ is said to be admissible if $\big\{u\in \mathcal{H}\, \bigl|\,gu=u, \forall g\in G\big\}=\{0\}$ and every $\sigma$-admissible and equivariant map $f:\overline{U}\rightarrow \mathcal{H}_0$, where $U$ is an open bounded invariant neighborhood of the origin in $\mathcal{H}$ such that $\overline{U}$ is $\sigma-$closed and $\mathcal{H}_0$ is a proper subspace of $\mathcal{H}$, has a zero on $\partial U$.
\end{defin}
 \begin{exa}
 Theorem $\ref{bu}$ implies that the antipodal action of $\mathbb{Z}_2 $ on every separable Hilbert space is admissible.
 \end{exa}

%
\section{A deformation lemma}
%
%
\indent Let $Y$ be a closed separable subspace of a Hilbert space $X$ endowed with the inner product $(\cdot)$ and the associated norm $\|\cdot \|$.
Let $P:X\rightarrow Y$ and $Q:X\rightarrow Y^{\perp}$ be the orthogonal projections.  Let $(\theta_j)$ be an orthonormal basis of $Y$. On $X$ we consider a new norm
\begin{equation}\label{e.2}
   \vvvert u\vvvert :=\max\Big(\sum_{j=0}^{\infty} \frac{1}{2^{j+1}} |(Pu, \, \theta_{j})  |,\|Q u\|\Big),
\end{equation}
 and we denote by $\tau$ the topology and all topological notions related to the topology generated by $\vvvert \cdot\vvvert$ \big(see \cite{K-S} or \cite{W}\big). It is clear that $\|Qu\|\leq\vvvert u\vvvert\leq\|u\|$. Moreover, if $(u_n)$ is a bounded sequence in $X$ then
\begin{equation*}
    u_n
\stackrel{\tau}{\rightarrow}u  \Longleftrightarrow Pu_n \rightharpoonup Pu \,\ and \,\ Qu_n \rightarrow Qu.
\end{equation*}
\indent We recall some standard notations:\\
 Let $S\subset X$ and $\varphi\in\mathcal{C}^1(X,{\mathbb R})$.  $dist(u,S):=\|u-S\|$, $dist_\tau(u,S):=\vvvert u-S\vvvert$, $S_\alpha:=\bigl\{u\in X\,\ \bigl| \,\ dist(u,S)\leq\alpha\bigr\}$, $\varphi^a:=\bigl\{u\in X\,\ \bigl| \,\ \varphi(u)\leq a\bigr\}$.\\
\indent The following lemma is somewhat a combination of Lemma $3.1$ and Lemma $6.8$ in \cite{W}.
\begin{lemme}[Deformation lemma]\label{d.l}
Assume that a finite group $G$ acts isometrically and $\tau$-isometrically on the Hilbert space $X$.
Assume also that $\varphi\in\mathcal{C}^1(X,{\mathbb R})$ is invariant and $\tau$-upper semicontinuous and $\nabla\varphi$ is weakly sequentially continuous. Let $S\subset X$ be invariant and $c\in{\mathbb R}$, $\epsilon,\delta>0$ such that
\begin{equation}\label{e.3}
    \forall u\in \varphi^{-1}([c-2\epsilon,c+2\epsilon])\cap S_{2\delta}, \, \|\varphi'(u)\|\geq \frac{8\epsilon}{\delta}.
\end{equation}
Then there exists $\eta \in\mathcal{C}([0,1]\times\varphi^{c+2\epsilon},X)$ such that:
\begin{itemize}
  \item [(i)] $\eta(t,u)=u$ if $t=0$ or if $u\notin\varphi^{-1}([c-2\epsilon,c+2\epsilon])\cap S_{2\delta},$
  \item [(ii)] $\eta(1,\varphi^{c+\epsilon}\cap S)\subset\varphi^{c-\epsilon},$
  \item [(iii)] $\|\eta(t,u)-u\|\leq \frac{\delta}{2}$ $\forall u\in \varphi^{c+2\epsilon},$ $\forall t\in[0,1],$
  \item [(iv)] $\varphi(\eta(\cdot,u))$ is non increasing, $\forall u\in \varphi^{c+2\epsilon}$,
  \item [(v)] Each point $(t,u)\in [0,1]\times\varphi^{c+2\epsilon}$ has a $\tau$-neighborhood $N_{(t,u)}$ such that $\big\{v-\eta(s,v)\, \bigl| \, (s,v)\in N_{(t,u)}\cap([0,1]\times\varphi^{c+2\epsilon})\bigr.\big\}$ is contained in a finite-dimensional subspace of $X$,
  \item [(vi)] $\eta$ is $\tau$-continuous,
  \item [(vii)] $\eta(t,\cdot)$ is equivariant $\forall t\in [0,1]$.
\end{itemize}
\end{lemme}
\begin{proof}
Let us define
\begin{equation*}
    w(v):=2\|\nabla \varphi(v)\|^{-2}\nabla \varphi(v), \,\,\,\,\ \forall v\in \varphi^{-1}\big([c-2\epsilon,c+2\epsilon]\big).
\end{equation*}
Since $\nabla \varphi$ is weakly sequentially continuous, for every $v\in \varphi^{-1}\big([c-2\epsilon,c+2\epsilon]\big)$ there exists a $\tau$-open invariant neighborhood $N_v$ of $v$ such that $(\nabla \varphi(u),w(v))>1$ $\forall u\in N_v$. Since $\varphi$ is $\tau$-upper semicontinuous, $\widetilde{N}:=\varphi^{-1}\big(]-\infty,c+2\epsilon[\big)$ is $\tau$-open. It follows then that the family
\begin{equation*}
    \mathcal{N}=\big\{N_v\, \bigl| \, c-2\epsilon\leq\varphi(v)\leq c+2\epsilon\bigr.\big\}\cup\widetilde{N}
\end{equation*}
is a $\tau$-open covering of $\varphi^{c+2\epsilon}$. Since $(\varphi^{c+2\epsilon},\tau)$ is metric, and hence paracompact, there exists a $\tau$-locally finite $\tau$-open covering $\mathcal{M}:=\big\{M_i:i\in I\big\}$ of $\varphi^{c+2\epsilon}$ finer than $\mathcal{N}$. Define
$$V:=\bigcup\limits_{i\in I}M_i.$$
For every $i\in I$ we have only the possibilities $M_i\subset N_v$ for some $v$ or $M_i\subset \widetilde{N}$. In the first case we define $v_i:=w(v)$ and in the second case $v_i:=0$. Let $\{\lambda_i\, \bigl | \, i\in I \bigr. \}$ be a $\tau$-Lipschitz continuous partition of unity subordinated to $\mathcal{M}$ and define on $V$
$$h(u):=\sum\limits_{i\in I}\lambda_i(u)v_i.$$
The map $h$ satisfies the following properties (see Lemma $6.7$ of \cite{W}):
\begin{itemize}
  \item [(a)] $h$ is $\tau$-locally Lipschitz continuous and locally Lipschitz continuous,
  \item [(b)] each point $u\in V$ has a $\tau$-neighborhood $V_u$ such that $h(V_u)$ is contained in a finite dimensional subspace of $X$,
  \item [(c)] $(\nabla\varphi(u),h(u))\geq0$ $\forall u\in V$,
  \item [(d)] $\forall u\in \varphi^{-1}\bigl([c-2\epsilon,c+2\epsilon]\bigr)$, $\bigl(\nabla\varphi(u),h(u)\bigr)>1$.
\end{itemize}
Let us define the equivariant vector field $\widetilde{h}$ on $V$ by:
 $$\widetilde{h}(u):=\frac{1}{|G|} \sum\limits_{g\in G}g^{-1}h(gu).$$
 We claim that $\widetilde{h}$ satisfies properties $(a)$, $(b)$, $(c)$ and $(d)$ above.
In fact, since $\varphi$ is invariant we have
\begin{equation*}
    \bigl(\nabla\varphi(gu),v\bigr)=\bigl(\nabla\varphi(u),g^{-1}v\bigr) \,\ \forall g\in G\,\ \forall u,v\in X.
\end{equation*}
Thus
\begin{align*}
  \bigl(\nabla\varphi(gu),\widetilde{h}(u)\bigr) &= \frac{1}{|G|}\sum\limits_{g\in G}\bigl(\nabla\varphi(u),g^{-1}h(gu)\bigr) \\
   &= \frac{1}{|G|}\sum\limits_{g\in G}\bigl(\nabla\varphi(gu),h(gu)\bigr),
\end{align*}
so that $\widetilde{h}$ satisfies $(c)$ and $(d)$. Let $u\in V$, for every $g\in G$ there exists a $\tau-$neighborhood $V_{gu}$ of $gu$ such that $h(V_{gu})$ is contained in a finite dimensional subspace of $X$, consequently $W_u:=\bigcup\limits_{g\in G}V_{gu}$ is a $\tau-$neighborhood of $u$ such that $\widetilde{h}(W_u)$ is contained in a finite dimensional subspace of $X$ and then $\widetilde{h}$ satisfies $(b)$. Finally $\widetilde{h}$ obviously satisfies $(a)$. \\

Let us define
\begin{equation*}
    A:=\varphi^{-1}\bigl([c-2\epsilon,c+2\epsilon]\bigr)\cap S_{2\delta},
\end{equation*}
\begin{equation*}
    B=:\varphi^{-1}\bigl([c-\epsilon,c+\epsilon]\bigr)\cap S_{\delta},
\end{equation*}
\begin{equation*}
    \psi(u):=dist_{\tau}\bigl(u,V\backslash A\bigr)\Big[dist_{\tau}\big(u,V\backslash A\big)+dist_{\tau}\big(u,B\big)\Big]^{-1} \,\,\,\,\,\ on \,\,\ V,
\end{equation*}
and
\begin{equation*}
    f(u):=\psi(u)\widetilde{h}(u), \,\,\,\ u\in V.
\end{equation*}
It is clear that $f$ is $\tau$-locally Lipschitz, $\tau$-continuous, continuous, locally Lipschitz and equivariant.
By assumption $(\ref{e.3})$, $\|w(u)\|\leq \frac{\delta}{4\epsilon}$, which implies $\|f(u)\|\leq \frac{\delta}{4\epsilon}$ on $V$.\\
For each $u\in \varphi^{c+2\epsilon}$, the Cauchy problem
$$\left\{
    \begin{array}{ll}
      \frac{d}{dt}\mu(t,u)=-f(\mu(t,u)) & \hbox{} \\
      \mu(0,u)=u  & \hbox{}
    \end{array}
  \right.
$$
has a unique solution $\mu(\cdot,u)$ defined on ${\mathbb R}^+$. Moreover $\mu$ is continuous on ${\mathbb R}^+\times\varphi^{c+2\epsilon}$.\\
Since $\|f(u)\|\leq \frac{\delta}{4\epsilon}$, we have for $t\geq0$
\begin{equation*}
    \|\mu(t,u)-u\|\leq\frac{\delta t}{4\epsilon}.
\end{equation*}
We also have
\begin{align*}
  \frac{d}{dt}\varphi(\mu(t,u)) &= \Big(\nabla\varphi(\mu(t,u)),\frac{d}{dt}\mu(t,u)\Big) \\
   &= -\Big(\nabla\varphi(\mu(t,u)),f(\mu(t,u)\Big)\\
   &= -\psi(\mu(t,u))\Big(\nabla\varphi(\mu(t,u)),\widetilde{h}(\mu(t,u)\Big)\leq0
\end{align*}
 It is easy to verify that if we define $\eta$ on $[0,1]\times \varphi^{c+2\epsilon}$ by $\eta(t,u):=\mu(2\epsilon t,u)$, then $(i)$,  $(ii)$, $(iii)$ and  $(iv)$ are satisfied.
The proof of $(v)$ and $(vi)$ is the same as that of $b)$ and $c)$ in Lemma $6.8$ of \cite{W} (with $T=2\epsilon$).
Since $f$ is equivariant, $(vii)$ is a direct consequence of the existence and the uniqueness of the solution of the above Cauchy problem.
\end{proof}

\section{An infinite-dimensional version of the fountain theorem}
%
%
In the sequel we assume that $ \displaystyle Z:=Y^\perp=\overline{\bigoplus_{j=0}^\infty\mathbb{R} e_j }$.\\
We use the following notation:
$$ \displaystyle Y_k:=Y\oplus(\bigoplus_{j=0}^k{\mathbb R} e_j), \,\,\,\,\,\,\,\,\,\,\,\,\,\,\ Z_k:=  \overline{\bigoplus_{j=k}^\infty\mathbb{R} e_j }, $$
$$ \displaystyle B_k:=\{u\in Y_k \, \bigl | \, ||u||\leq \rho_k \bigr.\}, \,\,\,\ N_k:=\{u\in Z_k \, \bigl | \, ||u||=r_k \bigr. \} \,\,\ where\,\,\,\ 0<r_k< \rho_k, \,\ k\geq2.   $$
 We denote $P_k:X\rightarrow Y_k$ and $Q_k:X\rightarrow Z_k$ the orthogonal projections. We define
\begin{equation*}
    e'_j:=\left\{
            \begin{array}{ll}
              e_j \,\,\,\ for \,\,\,\ j=0,1,...,k & \hbox{} \\
             \theta_{j-k-1} \,\,\,\ for \,\,\,\ j\geq k+1,  & \hbox{}
            \end{array}
          \right.
\end{equation*}
\big(where $(\theta_j)$ is the orthonormal basis of $Y$ we chose at the beginning of section $2$\big).  We can define a new norm on $X$ by setting
\begin{equation*}
   \vvvert u\vvvert_k=\max\Big(\sum\limits_{j=0}^{\infty}\frac{1}{2^{j+1}}|(P_ku,e'_j)|,||Q_{k+1}u||\Big).
   \end{equation*}
\begin{rem}\label{p.1}
\textnormal{Since a direct calculation shows that}
\begin{align*}
    \vvvert u\vvvert_k\leq\frac{3}{2}\vvvert u\vvvert \,\,\,\ \textnormal{and} \,\,\,\ \vvvert u \vvvert\leq2^{k+1}\vvvert u \vvvert _k,\,\,\, \forall u\in Y_k,
\end{align*}
\textnormal{then for every $k$, the norms $\vvvert \cdot \vvvert$ and $\vvvert \cdot\vvvert _k$ are equivalent on $Y_k$. In addition, it is easy to show that the projections $P_k$ are $\tau$-continuous, for every $k$.}
\end{rem}

\begin{lemme}[Intersection lemma]\label{i.l}
Under assumption $(A_1),$ let $\gamma \,
: \, B_{k} \rightarrow X $ such that
\begin{itemize}
  \item [(a)] $\gamma$ is equivariant and $\gamma \mid_{\partial B_{k}} = id,$
  \item [(b)] $\gamma$ is $\tau$-continuous,
  \item [(c)] every $u\in int (B_k)$ has a $\tau$-neighborhood $N_u$ in $Y_k$ such that $(id-\gamma)\bigl(N_u\cap int(B_k)\bigr)$ is contained in a finite-dimensional subspace of $X$.
\end{itemize}
 Then $\gamma (B_{k})
\cap N_{k} \neq \emptyset.$
\end{lemme}
\begin{proof}
Let $U=\{u\in B_k \, \bigl | \, ||\gamma(u)||<r_k\bigr.\}$. Since $\rho_k>r_k$ and $\gamma(0)=0$, $U$ is an open bounded and invariant neighborhood of $0$ in $Y_k$. It is clear that $B_k$ is $\tau$-closed, so we deduce from (b) that $\overline{U}$ is also $\tau$-closed. Consider the equivariant map $$P_{k-1}\gamma:\overline{U}\rightarrow Y_{k-1}.$$
\begin{enumerate}[(i)]
  \item $P_{k-1}\gamma$ is $\tau$-continuous. In fact, if $u_n\stackrel{\tau}{\rightarrow} u$, then from $(b)$ $\gamma(u_n)\stackrel{\tau}{\rightarrow}\gamma(u)$ and by Remark $\ref{p.1}$ we have
   $\vvvert P_{k-1}(\gamma(u_n)-\gamma(u))\vvvert \rightarrow 0$ as $n\rightarrow\infty$.
  \item Let $u\in U$. From $(c)$ $u$ has a $\tau$-neighborhood $N_u$ such that $(id-\gamma)(N_u\cap U)\subset W$, where $W$ is a finite-dimensional subspace of $X$. Let $v\in N_u\cap U\subset Y_k=Y_{k-1}\oplus {\mathbb R} e_k$, then $(id-P_{k-1}\gamma)(v)=P_{k-1}(v-\gamma(v))+\lambda e_k\in W+{\mathbb R} e_k$ which is finite-dimensional.
\end{enumerate}
Thus $P_{k-1}\gamma:\overline{U}\rightarrow Y_{k-1}$ is $\sigma-$admissible \big(in the sense of Definition $\ref{sigmaadm}$, where the pair $\big(\mathcal{H},(b_j)\big)$ in section $2$ is replaced by the pair $\big(Y_k,(e'_j)\big)$\big). Since the action of $G$ on $Y_k$ is admissible, there exists $u_0\in \partial U$ such that $P_{k-1}\gamma(u_0)=0$. This ends the proof of the lemma since $X=Y_{k-1}\oplus Z_k$.
\end{proof}
\begin{theo}\label{tm}
Under assumption $(A_1),$ let $\varphi\in\mathcal{C}^1(X,{\mathbb R})$ be invariant and $\tau$-upper semicontinuous such that
$\nabla\varphi$ is weakly sequentially continuous.
For $k \geq 2,$ define
\begin{equation*}
    a_{k} := \sup_{ \substack{ u \in Y_{k} \\ \|u \|= \rho_{k}
}} \varphi(u),
\end{equation*}
\begin{equation*}
    b_{k} := \inf_{ \substack{u \in Z_{k} \\ \| u \| = r_{k}}} \varphi(u),
\end{equation*}
\begin{equation*}
    c_{k} :=  \inf_{\gamma \in \Gamma_{k}} \sup_{u\in B_{k}} \varphi
\bigl( \gamma(u)  \bigr),
\end{equation*}
where
\begin{align*}
  \Gamma_{k}  &:=  \Big\{ \gamma:B_{k} \rightarrow X \,\ \bigl | \,\ \gamma \,\,\ \textnormal{is equivariant},\, \tau-\textnormal{continuous and } \gamma \mid_{\partial B_{k}} = id. \bigr.& \\
  &\textnormal{Every }u\in int (B_k)\, \textnormal{has a }\tau-\textnormal{neighborhood } N_u \, \textnormal{in } Y_k \, \textnormal{such that } (id-\gamma)\bigl(N_u\cap int(B_k)\bigr)\\
  & \textnormal{is contained in a finite-dimensional subspace of }X.\,\, \textnormal{Futhermore} \,\,\ \varphi(\gamma(u))\leq\varphi(u)\,\ \forall u\in B_k. \Big\},&
\end{align*}
and
\begin{equation*}
    d_k:=\displaystyle \sup_{\substack{u \in Y_k \\ \|u\| \leq \rho_k}} \varphi(u).
\end{equation*}

If
\begin{equation*}
    d_k<\infty \,\,\,\, \textnormal{and} \,\,\,\, b_k>a_k,
\end{equation*}

then $c_{k} \geq b_{k}$ and, for every $\epsilon \in \left]0, \,
(c_{k} - a_{k})/2 \right[,$ $\delta>0$ and $\gamma \in \Gamma_{k}$
such that
\begin{equation}\label{e.6}
    \sup_{B_{k}} \varphi \circ \gamma \leq c_{k}+\epsilon,
\end{equation}

there exists $u \in X$ such that
\begin{enumerate}
\item{$ c_{k}-2 \epsilon \leq \varphi (u) \leq c_{k} + 2 \epsilon,   $  }
\item{$\textnormal{dist}(u, \, \gamma(B_{k})) \leq 2 \delta,   $}
\item{$\|\varphi'(u)  \| \leq 8 \epsilon / \delta   .$   }
\end{enumerate}

\end{theo}
\begin{proof}
It follows from Lemma $\ref{i.l}$ that $c_k\geq b_k$. Assume by contradiction that the thesis is false. We apply Lemma $\ref{d.l}$ with $S:=\gamma(B_k)$. We may assume that
\begin{equation}\label{e.7}
    c_k-2\epsilon> a_k.
\end{equation}
We define on $B_k$ the map $\beta(u):=\eta(1,\gamma(u))$. We claim that $\beta\in\Gamma_k$.
\begin{enumerate}[(i)]
  \item  It follows from $(\ref{e.7})$ and $(i)$ of the deformation lemma  that $\beta_{|\partial B_k}=id$.
  \item  It is clear that $\beta$ is equivariant and $\tau$-continuous, and $\varphi(\beta(u))\leq \varphi(u)$ $\forall u\in B_k$.
  \item  Let $u\in int(B_k)$. Since $\gamma\in\Gamma_k$, $u$ has a $\tau$-neighborhood $N_u$ in $Y_k$ such that $(id-\gamma)\bigl(N_u\cap int(B_k)\bigr)\subset W_1$, where $W_1$
  is a finite-dimensional subspace of $X$. From $(v)$ of the deformation lemma the point $(1,\gamma(u))$ has a $\tau$-neighborhood $M_{(1,\gamma(u))}=M_1\times M_{\gamma u}$ such that $\big\{z-\eta(s,z)\, \bigl | \, (s,z)\in M_{(1,\gamma(u))}\cap ([0,1]\times \varphi^{c_k+2\epsilon})\big\}$ is contained in a finite-dimensional subspace $W_2$ of $X$. Thus for every $v\in N_u\cap \gamma^{-1}(M_{\gamma(u)})\cap B_k$, we have $(id-\beta)(v)=(id-\gamma)(v)+\gamma(v)-\eta(1,\gamma(v))\in W_1+W_2$ which is finite-dimensional.
\end{enumerate}
\indent Thus $\beta\in \Gamma_k$.\\
\indent Now by using $(\ref{e.6})$ and $(ii)$ of Lemma $\ref{d.l}$ we obtain
\begin{equation*}
  { \displaystyle c_k\leq\sup_{\substack{u \in Y_k \\ \|u\| = \rho_k}} \varphi(\beta(u)) = \displaystyle \sup_{\substack{u \in Y_k \\ \|u\| = \rho_k}} \varphi(\eta(1,\gamma(u))) \leq c_k-\epsilon},
\end{equation*}
which contradicts the definition of $c_k$.
\end{proof}
\begin{theo}\label{theomultsol}
Under assumption $(A_1),$ let $\varphi\in\mathcal{C}^1(X,{\mathbb R})$ be invariant and $\tau$-upper semicontinuous such that
$\nabla\varphi$ is weakly sequentially continuous. If there exists $\rho_k > r_k >0$ such that
\begin{enumerate}
\item[$(A_2)$]{$a_k \, := \, \displaystyle \sup_{\substack{u \in Y_k \\ \|u\| = \rho_k}} \varphi(u) \le 0 $ \,\,\ and \,\,\ $d_k:=\displaystyle \sup_{\substack{u \in Y_k \\ \|u\| \leq \rho_k}} \varphi(u)<\infty $}.
\item[$(A_3)$]{$b_k \, := \, \displaystyle \inf_{\substack{u \in Z_k \\ \|u \| = r_k}} \varphi(u) \to \infty, \, k \to \infty$.}
\end{enumerate}
Then there exists a sequence $(u_k^n)_n\subset X$ such that
\begin{align*}
\varphi'(u_k^n)  \to 0, & & \varphi(u_k^n)  \to c_{k} :=  \inf_{\gamma \in \Gamma_{k}} \sup_{u\in B_{k}} \varphi
\bigl( \gamma(u)  \bigr) \,\,\,\,\,\,\ as \,\ n\to\infty.
\end{align*}
\end{theo}
\begin{proof}
Choose $k$ sufficiently large and apply the preceding theorem.
\end{proof}
\indent Recall that the functional $\varphi$ satisfies the $(PS)_c$-condition \big(Palais-Smale condition at level $c$\big), if every sequence $(u_n)\subset X$ such that
\begin{equation*}
    \varphi(u_n)\rightarrow c \,\,\,\ \text{and} \,\,\,\ \varphi'(u_n)\rightarrow0\,\,\ \text{as} \,\,\ n\rightarrow\infty,
\end{equation*}
has a convergent subsequence.
\begin{cor}\label{ft.2}
Under the assumptions of the preceding theorem, if $\varphi$ satisfies in addition the $(PS)_c$ condition for every $c>0$, then
$\varphi$ has an unbounded sequence of critical values.
\end{cor}
\indent In the sequel $|\cdot|_p$ is the usual norm in $L^p.$
%
\section{Semilinear Schr\"{o}dinger equation}

In this section we apply our abstract theorem to the resolution of the semilinear Schr\"{o}dinger equation
\begin{equation}\label{semischrodeq}
    \left\{
      \begin{array}{ll}
        -\Delta u+V(x)u=f(x,u), & \hbox{} \\
        u\in H^1(\mathbb{R}^N). & \hbox{}
      \end{array}
    \right.
\end{equation}
We define the functional
\begin{equation}\label{eqfunctschrod}
    \varphi(u):=\frac{1}{2}\int_{\mathbb{R}^N}\bigl(|\nabla u|^2+V(x)u^2\bigr)dx-\int_{\mathbb{R}^N}F(x,u)dx, \,\,\,\,\,\ u\in H^1(\mathbb{R}^N).
\end{equation}
\indent It is well known that if $\varphi$ is of class $\mathcal{C}^1$ then its critical points are weak solutions of ($\ref{semischrodeq}$). Observe also that, due to the periodicity of $f$ and $v$, if $u$ is a solution of $(\ref{semischrodeq})$, then so is $g*u$ for each $g\in\mathbb{Z}^N$, where $(g*u)(x):=u(x+g)$. Two solutions $u$ and $v$ of ($\ref{semischrodeq}$) are said to be geometrically distinct if the sets $\big\{g*u \,\ \bigl|\,\ g\in\mathbb{Z}^N\big\}$ and $\big\{g*v\,\ \bigl|\,\ g\in\mathbb{Z}^N\big\}$ are disjoint. \\
\indent We will prove:
\begin{theo}\label{theoschrodinger}
Assume $(V_0)$, $(f_1)-(f_5)$. Then problem $(\ref{semischrodeq})$ has a sequence $(u_k)$ of solutions such that $\varphi(u_k)\rightarrow\infty$, as $k\rightarrow\infty.$
\end{theo}
\begin{rem}
Infinitely many of the solutions obtained in Theorem $\ref{theoschrodinger}$ above are geometrically distinct. In fact, since $\varphi(u_k)\rightarrow\infty,$ as $k\rightarrow\infty$, there exists $k_0>0$ big enough such that for every $i,j>k_0$, if $i\neq j$ then $\varphi(u_i)\neq\varphi(u_j).$
\end{rem}
Before giving the proof of Theorem $\ref{theoschrodinger}$ we need some preliminary results.\\
Let $L$ be the self-adjoint operator $L:H^1(\mathbb{R}^N)\rightarrow H^1(\mathbb{R}^N)$ defined by
 \begin{equation*}
   (Lu,v)_1:=\int_{\mathbb{R}^N}\bigl(\nabla u\cdot\nabla v+V(x)uv\bigr)dx,
 \end{equation*}
 \big(where ($\cdot)_1$ is the usual inner product in $H^1(\mathbb{R}^N)$\big). By assumption $(V_0)$, $X:=H^1(\mathbb{R}^N)$ is the sum of two infinite-dimensional $L-$invariant orthogonal subspaces $Y$ and $Z$ on which $L$ is respectively negative definite and positive definite (see \cite{K-S}). We denote $P:X\rightarrow Y$ and $Q:X\rightarrow Z$ the orthogonal projections. We introduce a new inner product on $X$ \big(equivalent to $(\cdot)_1$\big) by the formula
\begin{equation*}
    (u,v):=\bigl(L(Qu-Pu),v\bigr)_1, \,\,\,\ u,v\in X
\end{equation*}
with the corresponding norm
\begin{equation*}
    \|u\|:=(u,u)^{\frac{1}{2}}.
\end{equation*}
Since the inner products $(\cdot)$ and $(\cdot)_1$ are equivalent, $Y$ and $Z$ are also orthogonal with respect to $(\cdot)$.\\
One can verify easily that $(\ref{eqfunctschrod})$ reads
\begin{equation}\label{varphi}
    \varphi(u)=\frac{1}{2}\bigl(\|Qu\|^2-\|Pu\|^2\bigr)-\int_{\mathbb{R}^N}F(x,u)dx \,\,\ \forall u,v\in H^1(\mathbb{R}^N) .
\end{equation}
We refer the reader to \cite{K-S} or \cite{W} for the proof of the following lemma.
\begin{lemme}
Under assumptions $(V_0)$, $(f_1)-(f_3)$, $\varphi$ is of class $\mathcal{C}^1$ and is $\tau-$upper semicontinuous, and $\nabla\varphi$ is weakly sequentially continuous. Moreover we have
\begin{equation}\label{deriveevarphi}
    \big<\varphi'(u),v\big>=(Qu,v)-(Pu,v)-\int_{\mathbb{R}^N}vf(x,u)dx.
    \end{equation}
\end{lemme}
We will also need the following lemma due to Br\'{e}zis and Lieb \big(see Theorem $2$ in \cite{B-L}\big).
\begin{lemme}[Br\'{e}zis-Lieb, $1983$]\label{lemmeBL}
Let $J:\mathbb{C}\rightarrow\mathbb{C}$ be a continuous function such that $J(0)=0$ and for every sufficiently small $\epsilon>0 $ there exist two continuous, nonnegative functions $h_\epsilon$ and $g_\epsilon$ satisfying $$|J(a+b)-J(a)|\leq\epsilon h_\epsilon(a)+g_\epsilon(b), \,\ for \,\ all\,\ a,b\in\mathbb{C}.$$
Let $u_n=u+v_n$ be a sequence of measurable functions from $\mathbb{R}^N$ to $\mathbb{C}$ such that:
\begin{itemize}
  \item [(i)] $v_n\rightarrow0$ a.e.
  \item [(ii)] $J(u)\in L^1(\mathbb{R}^N)$.
  \item [(iii)] $\limsup_{n}\int_{\mathbb{R}^N}h_\epsilon(v_n)dx\leq C<\infty$, for some constant $C$ independent of $\epsilon$.
  \item [(iv)] $\int_{\mathbb{R}^N}g_\epsilon(u)dx<\infty$.
\end{itemize}
Then
$$\int_{\mathbb{R}^N}|J(u+v_n)-J(v_n)-J(u)|dx\rightarrow0 \,\ \textnormal{as}\,\ n\rightarrow\infty.$$
\end{lemme}

\begin{proof} [\textbf{Proof of theorem $\ref{theoschrodinger}$}]
\indent Let $(e_i)_{i\geq0}$ be an orthonormal basis of $(Z,\|\cdot\|)$ and let us define
\begin{equation*}
  \displaystyle  Y_k:=Y\oplus\big(\bigoplus_{i=0}^{k}\mathbb{R} e_i\big) \,\,\,\,\ and \,\,\,\,\ Z_k:=\overline{\bigoplus_{i=k}^{\infty}\mathbb{R} e_i}.
\end{equation*}
\indent Let $u=y+z\in Y_k$, with $y\in Y$ and $z\in\bigoplus\limits_{i=0}^{k}\mathbb{R} e_i$.  $(f_4)$ implies that for each $ \delta>0$ there exists $c_1=c_1(\delta)$ such that
\begin{equation*}
    F(x,u)\geq c_1|u|^\gamma-\delta|u|^2,
\end{equation*}
and then
\begin{equation*}
     -\int_{\mathbb{R}^N}F(x,u)dx\leq\delta|u|_2^2-c_1|u|_\gamma^\gamma.
\end{equation*}
 Since the norms $\|\cdot\|_1$ and $\|\cdot\|$ are equivalent, there exists a constant $c_2>0$ such that $|u|_2^2\leq c_2\|u\|^2=c_2\big(\|y\|^2+\|z\|^2\big)$. We then deduce that
\begin{equation*}
    \varphi(u)\leq \bigl(\delta c_2-\frac{1}{2}\bigr)\|y\|^2+\bigl(\frac{1}{2}+\delta c_2\bigr)\|z\|^2-c_1|u|_\gamma^\gamma.
\end{equation*}
Let $E_k$ be the closure of $Y_k$ in $L^\gamma(\mathbb{R}^N).$ We know that the Sobolev space $H^1(\mathbb{R}^N)$ embeds continuously in $L^\gamma(\mathbb{R}^N),$ then there exists a continuous projection of $E_k$ on $\bigoplus\limits_{i=0}^{k}\mathbb{R} e_i$, and since in a finite-dimensional vector space all norms are equivalent, there is a constant $c_3>0$ such that $c_3\|z\|\leq|u|_\gamma.$ Therefore we have
\begin{equation*}
    \varphi(u)\leq\bigl(\delta c_2-\frac{1}{2}\bigr)\|y\|^2+\bigl(\frac{1}{2}+\delta c_2\bigr)\|z\|^2-c_1c_3^\gamma\|z\|^\gamma.
\end{equation*}
If we choose $\delta$ such that $\delta c_2\leq\frac{1}{4}$ then
\begin{equation*}
    \varphi(u)\leq-\frac{1}{4}\|y\|^2+c_4\|z\|^2-c_5\|z\|^\gamma,
\end{equation*}
for some constants $c_4>0$ and $c_5>0$. This implies that
\begin{equation*}
    \varphi(u)\rightarrow-\infty \,\,\ \textnormal{as} \,\,\ \|u\|\rightarrow\infty.
\end{equation*}
Hence relation $(A_2)$ of Theorem $\ref{theomultsol}$ is satisfied for $\rho_k$ sufficiently large.\\
\indent Now let $u\in Z_k$, then $Pu=0$ and $Qu=u$.
Assumptions $(f_1)$, $(f_2)$ and $(f_3)$ imply that
\begin{equation}\label{bornef}
  \forall\epsilon>0,\exists c_\epsilon>0 \,\ \text{such  that} \,\  |f(x,u)|\leq \epsilon|u|+c_\epsilon|u|^{p-1},
\end{equation}
 hence
\begin{equation*}
    F(x,u)\leq\frac{\epsilon}{2}|u|^2+c'_\epsilon|u|^p,
\end{equation*}
and
\begin{equation*}
    \varphi(u)\geq\frac{1}{2}\|u\|^2-\frac{\epsilon}{2}|u|_2^2-c'_\epsilon|u|_p^p.
\end{equation*}
Let us define
\begin{equation*}
    \beta_k:=\sup_{\substack{v\in Z_k \\ \|v\|=1}}|v|_p
\end{equation*}
so that
\begin{equation*}
    \varphi(u)\geq\frac{1}{2}\|u\|^2-\frac{\epsilon}{2}|u|_2^2-c'_\epsilon\beta_k^p\|u\|^p\geq\frac{1}{2}(1-c_2\epsilon)\|u\|^2-c'_\epsilon\beta_k^p\|u\|^p.
\end{equation*}
Choosing $\epsilon=\frac{1}{2c_2}$ we obtain
\begin{equation*}
    \varphi(u)\geq\frac{1}{2}\bigl(\frac{1}{2}\|u\|^2-c\beta_k^p\|u\|^p\bigr).
\end{equation*}
Hence we have for $\|u\|=r_k:=(cp\beta_k^p)^{\frac{1}{2-p}},$
\begin{equation*}
    \varphi(u)\geq\frac{1}{2}\bigl(\frac{1}{2}-\frac{1}{p}\bigr)(cp\beta_k^p)^{\frac{2}{2-p}}.
\end{equation*}
We know by Lemma $3.8$ in \cite{W} that $\beta_k\rightarrow0$ as $k\rightarrow\infty$, hence relation $(A_3)$ of Theorem $\ref{theomultsol}$ is satisfied.\\
We apply Theorem $\ref{theomultsol}$ with the action of $\mathbb{Z}_2$ and we get the existence of a sequence $(v_k^n)_{n\geq0}$ in $X$ such that
\begin{align*}
\varphi(v_k^n)  \to c_k &\,\,\,\,\,\,\,\,\,\,\,\ \textnormal{and} & \varphi'(v_k^n)  \to 0 \,\,\,\,\,\,\ \textnormal{as} \,\ n\to\infty, \,\,\,\ \textnormal{for every}\,\ k .
\end{align*}
By Lemma $1.5$ of \cite{K-S} the sequence $(v_k^n)_n$ is bounded, and it is evident that for $k$ big enough no subsequence of $(v_k^n)_n$ converges to $0$. By Lemma $1.7$ of \cite{K-S}, there is a sequence $(a_n)\subset\mathbb{R}^N$ and numbers $r,\delta>0$ such that
\begin{equation*}
    \liminf_{\substack{n\rightarrow\infty}}\int_{B(a_n,r)}|v_k^n|^2dx\geq\delta, \,\,\ \text{for k big enough}.
\end{equation*}
Taking a subsequence if necessary we may suppose that, for $k$ big enough,
\begin{equation}\label{vbounded}
    ||v_k^n||_{L^2(B(a_n,r))}\geq \frac{\delta}{2}, \,\,\ \forall n.
\end{equation}
Choose $g_n\in\mathbb{Z}^N$ such that $|g_n-a_n|=\min\big\{|g-a_n|:g\in \mathbb{Z}^N\big\}$. Thus $|g_n-a_n|\leq \frac{1}{2}\sqrt{N}$. Define
\begin{equation}
    u_k^n:=g_n*v_k^n.
\end{equation}
In view of $(\ref{vbounded}),$ we have for $k$ big enough
\begin{equation}\label{ubounded}
    ||u_k^n||_{L^2(B(0,r+\frac{1}{2}\sqrt{N}))}\geq \frac{\delta}{2}, \,\,\ \forall n.
\end{equation}
It is not difficult to see that $\varphi(u_k^n)=\varphi(v_k^n)$ and $||\nabla\varphi(u_k^n)||=||\nabla\varphi(v_k^n)||$. Hence we have
\begin{align*}
\varphi(u_k^n)  \to c_k, & & \varphi'(u_k^n)  \to 0 \,\,\,\,\,\,\ \textnormal{as} \,\ n\to\infty .
\end{align*}
Again by Lemma $1.5$ of \cite{K-S} the sequence $(u_k^n)_n$ is bounded. Thus we deduce that, up to a subsequence,
\begin{align}
  u_k^n &\rightharpoonup u_k \,\ \text{in} \,\ X,\,\ \text{as} \,\ n\rightarrow\infty,& \\
  u_k^n &\rightarrow u_k \,\ \text{in} \,\ L_{loc}^2(\mathbb{R}^N),\,\ \text{as} \,\ n\rightarrow\infty,&\label{convinLloc}\\
  u_k^n &\rightarrow u_k \,\ \text{a.e. on} \,\ \mathbb{R}^N,\,\ \text{as} \,\ n\rightarrow\infty.&
\end{align}
 By $(\ref{ubounded})$ $u_k\neq0$ for $k$ big enough, and in view of the weak continuity of $\nabla\varphi$ we get that $u_k$ is a critical point of $\varphi$, and then a weak solution of $(\ref{semischrodeq})$.\\
\indent
Using $(\ref{varphi})$ and $(\ref{deriveevarphi})$ we have,
\begin{equation}\label{varphiu}
    \varphi(u_k^n)=\frac{1}{2}<\varphi'(u_k^n),u_k^n>+\frac{1}{2}\int_{\mathbb{R}^N}u_k^nf(x,u_k^n)dx-\int_{\mathbb{R}^N}F(x,u_k^n)dx.
\end{equation}
By $(\ref{convinLloc})$ we have for $0<R<\infty$
\begin{equation*}
    \sup_{\substack{y\in\mathbb{R}^N}}\int_{y+B_R}|u_k^n-u_k|^2\rightarrow0,\,\ \textnormal{as} \,\ n\rightarrow\infty,
\end{equation*}
then by Lemma $1.21$ of \cite{W} (see also \cite{PLL})
\begin{equation*}
    u_k^n\rightarrow u_k \,\ \textnormal{in} \,\ L^p(\mathbb{R}^N) \,\ \textnormal{as} \,\ n\rightarrow\infty.
\end{equation*}

 Now by $(\ref{bornef})$ for every $\epsilon>0$, there is $c_\epsilon>0$ such that
\begin{align*}
  \int_{\mathbb{R}^N}|(u_k^n-u_k)f(x,u_k^n-u_k)|dx &\leq\epsilon|u_k^n-u_k|_2^2+c_\epsilon|u_k^n-u_k|_p^p \,\,\,\ \textnormal{and}& \\
  \int_{\mathbb{R}^N}|F(x,u_k^n-u_k)|dx &\leq\frac{\epsilon}{2}|u_k^n-u_k|_2^2+\frac{c_\epsilon}{p}|u_k^n-u_k|_p^p \,\,\,\ \forall n.&
\end{align*}
  Thus we have
  \begin{equation}\label{convuandF}
  \left\{
    \begin{array}{ll}
     \int_{\mathbb{R}^N}(u_k^n-u_k)f(x,u_k^n-u_k)dx \rightarrow 0 \,\,\ \textnormal{as} \,\ n\rightarrow\infty & \hbox{} \\ \\
     \int_{\mathbb{R}^N}F(x,u_k^n-u_k)dx\rightarrow0 \,\,\ \textnormal{as} \,\ n\rightarrow\infty.  & \hbox{}
    \end{array}
  \right.
  \end{equation}
  One can easily verify that $(\ref{bornef})$ implies that, for almost every $x\in\mathbb{R}^N$, the functions $s\mapsto sf(x,s)$ and $s\mapsto F(x,s)$ satisfy the conditions of Lemma $\ref{lemmeBL}$, with $u_n\equiv u_k^n$, $u\equiv u_k$ and $v_n=u_k^n-u_k$. It then follows from Lemma $\ref{lemmeBL}$ and $(\ref{convuandF})$ that
   \begin{equation}\label{convuandF2}
  \left\{
    \begin{array}{ll}
     \int_{\mathbb{R}^N}u_k^nf(x,u_k^n)dx \rightarrow \int_{\mathbb{R}^N}u_kf(x,u_k)dx \,\,\ \textnormal{as} \,\ n\rightarrow\infty & \hbox{} \\ \\
     \int_{\mathbb{R}^N}F(x,u_k^n)dx\rightarrow \int_{\mathbb{R}^N}F(x,u_k)dx \,\,\ \textnormal{as} \,\ n\rightarrow\infty.  & \hbox{}
    \end{array}
  \right.
  \end{equation}
 Taking the limit $n\rightarrow\infty$ in $(\ref{varphiu})$ and using $(\ref{convuandF2})$ we obtain
 \begin{equation*}
    c_k=\frac{1}{2}\int_{\mathbb{R}^N}u_kf(x,u_k)dx-\int_{\mathbb{R}^N}F(x,u_k)dx.
 \end{equation*}
Now $(\ref{varphi})$ and $(\ref{deriveevarphi})$ also implies, since $\varphi'(u_k)=0$, that
\begin{equation*}
    \varphi(u_k)=\frac{1}{2}\int_{\mathbb{R}^N}u_kf(x,u_k)dx-\int_{\mathbb{R}^N}F(x,u_k)dx=c_k.
\end{equation*}
This ends the proof of the theorem, since $c_k\rightarrow\infty$ as $k\rightarrow\infty.$
\end{proof}
\begin{rem}
\textnormal{The existence of infinitely many geometrically distinct solutions of $(\ref{semischrodeq})$, under $(V_0)$, $(f_1)-(f_5)$, was first proved by Kryszewski and Szulkin in \cite{K-S} by using the degree we present before and a variant of Benci's pseudoindex \cite{Ben}. However they assumed in addition that there are $\lambda>0$ and $R>0$ such that}
\begin{equation*}
    |f(x,u+v)-f(x,u)|\leq\lambda|v|(1+|u|^{p-1}),\,\ \forall x\in\mathbb{R}^N, \,\ \forall u,v\in\mathbb{R}, \,\ \textnormal{with}\,\ |v|\leq R.
\end{equation*}
\end{rem}
%
\section{Noncooperative elliptic system}
In this section we apply our abstract result to the resolution of the following potential system
\begin{equation}\label{pt}
    \left\{
                                          \begin{array}{ll}
                                            \,\,\  \Delta u=H_u(x,u,v) \,\ \text{in} \,\ \Omega, & \hbox{} \\
                                            -\Delta v=H_v(x,u,v) \,\ \text{in} \,\ \Omega, & \hbox{} \\
                                          \,\,\  u=v=0 \,\ \text{on} \,\ \partial\Omega, & \hbox{}
                                          \end{array}
                                        \right.
\end{equation}
where $\Omega$ is an open bounded subset of $\mathbb{R}^N$.\\
\indent On $H_0^1(\Omega)$ we choose the norm $\|u\|:=|\nabla u|_2$, which by the Poincar\'{e} inequality is equivalent to the usual norm of $H_0^1(\Omega)$. On the space $X:=H_0^1(\Omega)\times H_0^1(\Omega)$, we choose the product norm $\|(u,v)\|=\sqrt{\|u\|^2+\|v\|^2}$ and we define the functional
\begin{equation}
    \Phi(u,v):=\int_\Omega \Bigl(\frac{1}{2}|\nabla v|^2- \frac{1}{2}|\nabla u|^2-H(x,u,v)\Bigr)dx.
\end{equation}
It is well known that if $\Phi$ is of class $\mathcal{C}^1$, its critical points are weak solutions of $(\ref{pt})$.
Our main result in this section is stated as follows:

\begin{theo}\label{theopotential}
Assume that $(h_1)-(h_4)$ are satisfied.
Then $(\ref{pt})$ has a sequence of solutions $(u_k,v_k)$ in $H_0^1(\Omega)\times H_0^1(\Omega)$ such that $\Phi(u_k,v_k)\rightarrow\infty$ as $k\rightarrow\infty$.
\end{theo}
The following lemma is well known (see for example \cite{B-C} or \cite{F}).
\begin{lemme}
Under assumptions $(h_1)$ and $(h_2)$, $\Phi\in\mathcal{C}^1(X,\mathbb{R})$ and
\begin{equation}
    \big<\Phi'(u,v),(k,h)\big>=\int_\Omega\Big(\nabla v\nabla h-\nabla u\nabla k-kH_u(x,u,v)-hH_v(x,u,v)\Big)dx.
\end{equation}
\end{lemme}
Define
\begin{align*}
  Y: &=\bigl\{(u,0)\,\ \bigl|\,\ u\in H_0^1(\Omega)\bigr\},&  \\
  Z: &=\bigl\{(0,v)\,\ \bigl|\,\ v\in H_0^1(\Omega)\bigr\},&
\end{align*}
so that $X=Y\oplus Z$ and
\begin{equation}\label{functPhi}
    \Phi(u,v):=\frac{1}{2}||(0,v)||^2-\frac{1}{2}||(u,0)||^2-\int_\Omega H(x,u,v)dx.
\end{equation}
The proof of the following lemma follows the lines of the proof of Theorem $A.2$ in \cite{W}.
\begin{lemme}
\label{lemmeboundG}
Assume that $|\Omega|<\infty$, $1\leq p,r<\infty$ and $G\in\mathcal{C}(\overline{\Omega}\times\mathbb{R}\times\mathbb{R})$ such that
\begin{equation*}
    |G(x,u,v)|\leq c\bigl(1+|u|^{\frac{p}{r}}+|v|^{\frac{p}{r}}\bigr).
\end{equation*}
Then $\forall u,v\in L^p(\Omega)$, $G(\cdot,u,v)\in L^r(\Omega)$ and the operator $A:L^p(\Omega)\times L^p(\Omega)\rightarrow L^r(\Omega),(u,v)\mapsto G(x,u,v)$ is continuous.
\end{lemme}
\begin{lemme}
Under assumption $(h_1)$, $\Phi$ is $\tau-$upper semicontinuous and $\nabla \Phi$ is weakly sequentially continuous.
\end{lemme}

\begin{proof}
\begin{enumerate}[(i)]
  \item{Let $(u_n,v_n)\in X$ such that $(u_n,v_n)\stackrel{\tau}{\rightarrow}(u,v)$ and $c\leq \Phi(u_n,v_n)$. By the definition of $\tau$ we have $v_n\rightarrow v$ in $H_0^1(\Omega)$ and then $(v_n)$ is bounded.
Noting that $c\leq \Phi(u_n,v_n)$ and $H(x,u_n,v_n)\geq0$, then $(u_n)$ is bounded and $u_n\rightharpoonup u$ in $H_0^1(\Omega)$. Now since the embedding $H_0^1(\Omega)\hookrightarrow L^2(\Omega)$ is compact, $v_n\rightarrow v$ and $u_n\rightarrow u$ in $L^2(\Omega)$. Thus up to a subsequence $v_n(x)\rightarrow v(x)$ and $u_n(x)\rightarrow u(x)$ a.e on $\Omega$, and by $(h_1)$ $H(x,u_n(x),v_n(x))\rightarrow H(x,u(x),v(x))$ a.e on $\Omega$. It then follows from Fatou's Lemma and the weak lower semicontinuity of the norm $\|\cdot\|$ that $c\leq\Phi(u,v)$, and $\Phi$ is $\tau-$upper semicontinuous.}
  \item{Let $(u_n,v_n)\in X$ such that $(u_n,v_n)\rightharpoonup (u,v)$, then by Rellich theorem $u_n\rightarrow u$ and $v_n\rightarrow v$ in $L^p(\Omega)$. By the H\"{o}lder inequality \\
    $$ \Big|\int_\Omega \bigl(hH_u(x,u_n,v_n)+kH_v(x,u_n,v_n)-hH_u(x,u,v)-kH_v(x,u,v)\bigr)dx\Big|\leq $$
$$ \int_\Omega |h||H_u(x,u_n,v_n)-H_u(x,u,v)|dx+\int_\Omega |k||H_v(x,u_n,v_n)-H_v(x,u,v)|dx\leq $$
 $$|h|_p|H_u(x,u_n,v_n)-H_u(x,u,v)|_{\frac{p}{p-1}}+ |k|_p|H_v(x,u_n,v_n)-H_v(x,u,v)|_{\frac{p}{p-1}}.$$
It follows from Lemma $\ref{lemmeboundG}$ that
\begin{equation*}
    \int_\Omega \bigl(hH_u(x,u_n,v_n)+kH_v(x,u_n,v_n)-hH_u(x,u,v)-kH_v(x,u,v)\bigr)dx\rightarrow 0\,\ \textnormal{as} \,\ n\rightarrow\infty,
\end{equation*}
and then
\begin{equation*}
    \bigl(\nabla\Phi(u_n,v_n),(h,k)\bigr)\rightarrow\bigl(\nabla\Phi(u,v),(h,k)\bigr) \,\ \forall (h,k)\in X.
\end{equation*}
This shows that $\nabla\Phi$ is weakly sequentially continuous.}
\end{enumerate}
\end{proof}

\begin{lemme}\label{lemmeconvsubseq}
Under assumptions $(h_1)-(h_3)$, every sequence $(u_n,v_n)\subset X$ such that
\begin{equation*}
    d:=\sup_{\substack{n}}\Phi(u_n,v_n)<\infty \,\,\ \textnormal{and} \,\,\ \Phi'(u_n,v_n)\rightarrow0
\end{equation*}
has a convergent subsequence.
\end{lemme}
\begin{proof}
We know by \cite{Szul} that $(h_3)$ implies:
\begin{equation}\label{borneH}
    \exists a_1,a_2>0 \,\ \text{such that}\,\ H(x,u,v)\geq a_1\big(|u|^p+|v|^p\big)-a_2.
\end{equation}
For $n$ big enough we have
\begin{eqnarray*}
  d+||(u_n,v_n)|| &\geq& \Phi(u_n,v_n)-\frac{1}{2}\big<\Phi'(u_n,v_n),(u_n,v_n)\big> \\
   &=& \int_\Omega\frac{1}{2}\bigl(u_nH_u(x,u_n,v_n+v_nH_v(x,u_n,v_n)\bigr)dx-\int_\Omega H(x,u_n,v_n)dx \\
   &\geq& \bigl(\frac{p}{2}-1\bigr)\int_\Omega H(x,u_n,v_n)dx  \\
   &\geq& \bigl(\frac{p}{2}-1\bigr)\big[a_1(|u_n|_p^p+|v_n|_p^p)-c_2|\Omega|\big].
\end{eqnarray*}
This implies that
\begin{equation}\label{z1}
     |u_n|_p^p+|v_n|_p^p\leq C_1\|(u_n,v_n)\|+C_2,
\end{equation}
for some positive constants $C_1$ and $C_2$.\\
On the other hand, we have for $n$ big enough
\begin{equation*}
    \|v_n\|^2-\int_\Omega v_nH_v(x,u_n,v_n)dx=\big<\Phi'(u_n,v_n),(0,v_n)\big>\leq\|v_n\|,
\end{equation*}
and
\begin{equation*}
    \|u_n\|^2+\int_\Omega u_nH_u(x,u_n,v_n)dx=\big<-\Phi'(u_n,v_n),(u_n,0)\big>\leq\|u_n\|.
\end{equation*}
We then obtain, by using $(h_2)$
\begin{equation}\label{z2}
     \|u_n\|^2+\|v_n\|^2\leq \|u_n\|+\|v_n\|+c\bigl(|u_n|_p^p+|v_n|_p^p\bigr)+c.
\end{equation}

Using  $(\ref{z1})$ and $(\ref{z2})$, we deduce that
\begin{equation*}
    \|(u_n,v_n)\|^2\leq D_1\|(u_n,v_n)\|+D_2,
\end{equation*}
for some positive constants $D_1$ and $D_2$.
Thus $(u_n,v_n)$ is bounded. Up to a subsequence there exists $(u,v)\in X$ such that $(u_n,v_n)\rightharpoonup (u,v)$. By Rellich theorem $(u_n,v_n)\rightarrow (u,v)$ in $L^p(\Omega)\times L^p(\Omega)$, and by Lemma $\ref{lemmeboundG}$, $H_u(x,u_n,v_n)\rightarrow H_u(x,u,v)$ and $H_v(x,u_n,v_n)\rightarrow H_v(x,u,v)$ as $n\rightarrow\infty.$  \\
Now one can verify easily that
$$\|u_n-u\|^2 = -\big<\Phi'(u_n,v_n)-\Phi'(u,v),(u_n-u,0)\big>-\int_\Omega (u_n-u)(H_u(x,u_n,v_n)-H_u(x,u,v)) ,$$
$$\|v_n-v\|^2 =  \big<\Phi'(u_n,v_n)-\Phi'(u,v),(0,v_n-v)\big>+\int_\Omega (v_n-v)(H_u(x,u_n,v_n)-H_u(x,u,v)).$$
It is clear that $\big<\Phi'(u_n,v_n)-\Phi'(u,v),(u_n-u,0)\big>\rightarrow0$ as $n\rightarrow\infty$. By the H\"{o}lder inequality and Lemma $\ref{lemmeboundG}$,
$$\bigl|\int_\Omega (u_n-u)(H_u(x,u_n,v_n)-H_u(x,u,v))\bigr|\leq|u_n-u|_p|H_u(x,u_n,v_n)-H_u(x,u,v)|_{\frac{p}{p-1}}\rightarrow0.$$
Thus we have proved that $\|u_n-u\|\rightarrow0$. By the same way $\|v_n-v\|\rightarrow0$.
\end{proof}

\begin{proof}[\textbf{Proof of Theorem $\ref{theopotential}$}]
Let $(e_j)$ be an orthonormal basis of $H_0^1(\Omega)$ and define
\begin{equation*}
 \displaystyle   Y_k:=Y\oplus\bigl(\big\{0\big\}\times\bigoplus_{j=0}^k\mathbb{R}e_j\bigr) \,\,\,\,\,\,\,\ Z_k:=\big\{0\big\}\times\overline{\bigoplus_{j=k}^\infty\mathbb{R}e_j}.
\end{equation*}
\begin{enumerate}[(i)]
\item{Let $(u,v)\in Y_k$. By definition of $Y_k$, $v\in \bigoplus\limits_{j=0}^k\mathbb{R}e_j$, and by $(\ref{functPhi})$
\begin{equation*}
    \Phi(u,v)=\frac{1}{2}\|v\|^2-\frac{1}{2}\|u\|^2-\int_\Omega H\big(x,u,v\big)dx.
\end{equation*}
$(\ref{borneH})$ then implies
\begin{align*}
  \Phi(u,v) &\leq \frac{1}{2}\|v\|^2-\frac{1}{2}\|u\|^2-a_1|u|_p^p-a_1|v|_p^p+a_2|\Omega| & \\
   &\leq \frac{1}{2}\|v\|^2-\frac{1}{2}\|u\|^2-a_1|v|_p^p+a_2|\Omega|.&
\end{align*}
Since on the space $\oplus_{j=0}^k\mathbb{R} e_j$ all norms are equivalent, there exists a constant $c>0$ such that $c\|v\|^p\leq|v|_p^p$ and hence
\begin{equation*}
    \Phi(u,v)\leq -\frac{1}{2}\|u\|^2+\frac{1}{2}\|v\|^2-a'_1\|v\|^p+a_1|\Omega|.
\end{equation*}
This shows that $\Phi(u,v)\rightarrow-\infty$ as $\|(u,v)\|\rightarrow\infty$, so condition $(A_2)$ of Theorem $\ref{theomultsol}$ is satisfied for $\rho_k$ large enough.}
\item{Let $(0,v)\in Z_k$, then by $(\ref{functPhi})$
\begin{equation*}
    \Phi(0,v)=\frac{1}{2}||v||^2-\int_\Omega H(x,0,v)dx.
\end{equation*}
We deduce from $(h_2)$ the existence of $c>0$ such that
\begin{equation*}
    |H(x,u,v)|\leq c\big(1+|u|^p+|v|^p\big),
\end{equation*}
which implies that
\begin{equation*}
    \Phi(0,v)\geq\frac{1}{2}||v||^2-c|v|_p^p-c|\Omega|.
\end{equation*}
Define
\begin{equation*}
 \displaystyle   \beta_k:=\sup_{\substack{u\in \overline{\bigoplus_{j=k}^\infty\mathbb{R} e_j}\\ \|u\|=1}}|u|_p.
\end{equation*}
Then we have for
\begin{equation*}
   \|v\|= r_k:=(cp\beta_k^p)^{\frac{1}{2-p}},
\end{equation*}
\begin{equation*}
    \Phi(0,v)\geq (\frac{1}{2}-\frac{1}{p})(cp\beta_k^p)^{\frac{2}{2-p}}-c|\Omega|.
\end{equation*}
We know by  Lemma $3.8$ in \cite{W} that $\beta_k\rightarrow0$ as $k\rightarrow\infty$, so $\Phi(0,v)\rightarrow\infty$ as $k\rightarrow\infty$ and condition $(A_3)$ of Theorem $\ref{theomultsol}$ is satisfied. By Lemma $\ref{lemmeconvsubseq}$, $\Phi$ satisfies the Palais-Smale condition and by $(h_5)$ $\Phi$ is even.}
\end{enumerate}
 We then conclude by applying Corollary $\ref{ft.2}$ with the action of $\mathbb{Z}_2.$
\end{proof}

\section*{Conclusion}
In this paper, we presented a generalization of the Fountain Theorem to strongly indefinite functionals. The use of the $\tau-$topology introduced by Kryszewski and Szulkin permitted an extension of the Borsuk-Ulam Theorem to $\sigma-$admissible functions making the above-mentioned generalization quite natural. We believe that the ideas presented in this paper could be used to similarly generalize a result like the dual version of the Fountain Theorem (see \cite{W}, Theorem $3.18$ for instance). An adaptation of these ideas to separable reflexive Banach spaces will be also the subject of future research.
%

%
%
\end{document}